\documentclass[a4paper, 12pt]{amsart}
\numberwithin{equation}{section}
\usepackage{amssymb}
\usepackage{amsmath}
\usepackage{amsfonts}
\usepackage{marvosym}
\usepackage{tikz}
\newcommand*\circled[1]{\tikz[baseline=(char.base)]
{\node[shape=circle,draw,inner sep=1] (char) {#1};}}

\newtheorem{thm}{Theorem}[section]
\newtheorem{lem}[thm]{Lemma}

\newtheorem{cor}[thm]{Corollary}
\newtheorem{prop}[thm]{Proposition}

\theoremstyle{definition}

\newtheorem{defin}[thm]{Definition}

\def\cb{{\mathcal B}}

\def\ce{{\mathcal E}}

\def\ch{{\mathcal H}}

\def\cp{{\mathcal P}}

\def\cs{{\mathcal S}}

\def\cu{{\mathcal U}}

\def\ga{{\mathfrak A}}
\def\gb{{\mathfrak B}}
\def\gc{{\mathfrak C}}

\def\bc{{\mathbb C}}

\def\bn{{\mathbb N}}

\def\bz{{\mathbb Z}}

\def\b{\beta}
\def\g{\gamma}  
  
\def\eeps{\epsilon}
\def\eps{\varepsilon}

\def\p{\pi}

\def\f{\varphi}  
\def\th{\theta} 
\def\om{\omega}

\def\id{\hbox{id}}

\def\aut{\mathop{\rm aut}}

\def\id{{\rm id}}

\def\ad{\mathop{\rm ad}}

\def\idd{{1}\!\!{\rm I}}

\DeclareMathAlphabet{\mathpzc}{OT1}{pzc}{m}{it}

\begin{document}

\title[On $C^*$-norms on $\bz_2$-graded tensor products]
{On $C^*$-norms on $\bz_2$-graded tensor products}
\author{Vitonofrio Crismale}
\address{Vitonofrio Crismale\\
Dipartimento di Matematica\\
Universit\`{a} degli studi di Bari\\
Via E. Orabona, 4, 70125 Bari, Italy}
\email{\texttt{vitonofrio.crismale@uniba.it}}
\author{Stefano Rossi}
\address{Stefano Rossi\\
Dipartimento di Matematica\\
Universit\`{a} degli studi di Bari\\
Via della E. Orabona, 4, 70125 Bari, Italy} \email{{\tt
stefano.rossi@uniba.it}}
\author{Paola Zurlo}
\address{Paola Zurlo\\
Dipartimento di Matematica\\
Universit\`{a} degli studi di Bari\\
Via E. Orabona, 4, 70125 Bari, Italy}
\email{\texttt{paola.zurlo@uniba.it}}

\begin{abstract}

We systematically investigate $C^*$-norms on the algebraic graded product of $\bz_2$-graded $C^*$-algebras.
This requires to single out the notion of a compatible norm, that is a norm with respect to which the product grading is bounded.
We then focus  on the spatial norm proving that it is minimal among all compatible $C^*$-norms.
To this end, we  first show that commutative $\bz_2$-graded $C^*$-algebras enjoy a nuclearity property in the category
of graded $C^*$-algebras. In addition, we provide a characterization of the extreme even states of a given
graded $C^*$-algebra in terms of their restriction to its even part.

\end{abstract}

\maketitle
\section{Introduction}\label{sec1}

Given two $C^*$-algebras $\ga$ and $\gb$, their algebraic tensor product $\ga\odot\gb$ can easily be endowed with
a natural structure of $^*$-algebra. However, it is a well-known fact that in general $\ga\odot\gb$
can be completed to a $C^*$-algebra in several different ways. Phrased differently,
more than one $C^*$-norm can be introduced on $\ga\odot\gb$. As is known, among the possible norms, the so-called maximal and minimal norms play a privileged role
in that any other $C^*$-norm is bounded between this two, see \emph{e.g.} \cite{T1}.
However, when at least one of the two given $C^*$-algebras is nuclear, there is by definition precisely one way to complete
 $\ga\odot\gb$, namely all $C^*$-norms on it agree with each other.
For instance, commutative $C^*$-algebras and approximately finite-dimensional $C^*$-algebras are all examples of nuclear algebras, see \cite{T1}.
Tensor products of $C^*$-algebras is an accomplished theory where virtually anything is known,
although  further pieces of information have been added  until very recent times, see {\it e.g.}
\cite{OP, Pis, Pisbook}. 

When one starts adding more structure, such as a grading, possibly novel aspects  may and will occur.
A case in point is given by $\bz_2$-graded $C^*$-algebras, which in the physicists' parlance are often
referred to as superalgebras, {\it cf.} \cite{L}.
In particular, graded tensor products of $\bz_2$-graded $C^*$-algebras have recently been given a good
deal of attention in \cite{CDF}, where a definition of  quantum detailed balance for product systems endowed with 
a $\bz_2$-grading has been proposed.
Among other things,  the aforementioned paper provides
an in-depth analysis of the maximal norm on the graded tensor products of $\bz_2$-graded $C^*$-algebras,
as well as showing that a product state on the algebraic Fermi product of
two $\bz_2$-graded $C^*$-algebras is well defined whenever just one of the two states is even, namely it is invariant
with respect to the grading. This last result is actually a generalization of the
analogous result obtained in \cite{AM1} for the pivotal example of
the CAR algebra, see also \cite{BR2} and \cite{Fid2}.

In this paper, instead, much of the attention is lavished on the so-called spatial norm. In particular, we prove that it is
the smallest of all compatible norms.  To make this statement more precise,
we need to single out the notion of  a compatible norm on the algebraic Fermi product. This is by definition a norm
 such that the grading can be extended to  its completion. As a matter of fact, 
the problem of deciding wheter any $C^*$-norm
on a $\bz_2$-graded product has revealed delicate to handle, not least because
producing  counterexamples turns out to be as delicate.  Indeed, 
norms other than the maximal and the minimal one  are typically obtained by means of abstract methods, and therefore are difficult
to explicitly compute with. 

Our analysis of compatible norms also requires to make intensive use of  
extreme even states, namely the extreme points of the
(weakly*) compact convex set of all even states. Extreme even states for $\bz_2$-graded
$C^*$-algebras parallel pure states for general $C^*$-algebras. Indeed, they are sufficiently many to
separate the elements of the $C^*$-algebras. Furthermore, they return the set of all  pure states
when the grading is trivial.

The paper is organized as follows.
In Section \ref{staralg} we set the notation and recall the
basic notions on graded $C^*$-algebras and their graded tensor products, which we also refer to as
Fermi tensor products, as is done in \cite{CDF}.  We then focus on even states
on a given graded $C^*$-algebra.  In particular, we also provide a description of the
extreme even states of a commutative graded $C^*$-algebra. We end the section
by showing that a product state is even if and only if both of its marginal states are.\\
In Section \ref{hilbert} we consider $\bz_2$-graded Hilbert spaces and
their Fermi tensor product as a spatial counterpart of the construction for abstract
$C^*$-algebras. We introduce a notion of Fermi product of grading-equivariant
representations and show that the GNS representation of the product state of two
even states is nothing but the Fermi product of the  two GNS representations.\\
Finally, in Section \ref{spnorm}, after recalling the definition of the spatial
norm, we first show that it is a cross norm, as is the maximal norm.
We then move on to prove that the maximal and minimal norm are compatible with the grading.
In Proposition \ref{nuc} we prove that the Fermi product of two $C^*$-algebras can be endowed with only one $C^*$-norm
when one of the two is commutative.
In Theorem \ref{minimal} we prove that the spatial norm is actually minimal among all compatible norms.

As an outlook for the foreseeable future, we believe the present study might be a first step
towards addressing distributional symmetries on probability spaces based on graded $C^*$-algebras, as has already
been done for the CAR algebra in \cite{CFCMP, CrFid, CrFid2}, and for any $C^*$-algebra (with trivial grading) in
\cite{St}, where it is the spatial norm to play a key role.

\section{on tensor products of $*$-algebras}\label{staralg}
In this section, we collect some results on $\bz_2$-graded algebraic structures obtained as tensor product of graded $*$-algebras.
We first observe that, if not otherwise specified, throughout the paper all the structures we deal with will be taken unital.\\
If $\gb\subset\ga$ is an inclusion of unital $C^*$-algebras with a common unity,
the unital linear mapping $E:\ga\to \gb$ is called a projection if $E(b)=b$ for all $b\in\gb$,
and is said to be a $\gb$-bimodule map if $E(ab)=E(a)b$, and $E(ba)=bE(a)$ for all $a\in\ga, b\in\gb$.
A positive $\gb$-bimodule projection $E$ is called a \emph{conditional expectation}.\\
Consider the $C^*$-algebras $\ga_1$ and $\ga_2$, and denote by $\ga_1\otimes \ga_2$ the algebraic tensor product $\ga_1\odot \ga_2$ with the product and involution given by
$$
(a_1\otimes a_2)(a_1'\otimes a_2'):=a_1a_1'\otimes a_2a_2'\,,\quad (a_1\otimes a_2)^*:=a_1^*\otimes a_2^*\,,
$$
for all $a_1,a_1'\in\ga_1$, $a_2,a_2'\in\ga_2$. Let us denote by $\ga_1\otimes_{\max} \ga_2$ and $\ga_1\otimes_{\min} \ga_2$ the completion of $\ga_1\otimes \ga_2$ with respect to the maximal and minimal $C^*$-cross norm, respectively \cite{T1}.\\
Let $\cs(\ga)$ be the weak*-compact collecting the states on a $C^*$-algebra $\ga$. If one takes $\om_1\in\cs(\ga_1)$ and $\om_2\in\cs(\ga_2)$, their product state $\psi_{\om_1,\om_2}\in\cs(\ga_1\otimes_{\min} \ga_2)$ is well defined also on $\ga_1\otimes_{\max} \ga_2$, and consequently the notation $\psi_{\om_1,\om_2}\in\cs(\ga_1\otimes \ga_2)$ will be used in the sequel.\\

Consider now $\bz_2=\{1,-1\}$ with the product as the group operation, and a $*$-algebra $\ga$. The latter is called an {\it involutive $\bz_2$-graded algebra} if it decomposes as
$$
\ga=\ga_1\oplus\ga_{-1}
$$
and
$$
(\ga_i)^*=(\ga^*)_i\,,\,\,
\ga_i\ga_j\subset\ga_{ij}\,,\quad i,j=1,-1\,.
$$
The subspaces $\ga_i$, $i=1,2$ are called the homogeneous components of $\ga$, and correspondingly any element of $\ga_i$ is called a homogeneous element of $\ga$. For any homogeneous element $x\in\ga_{\pm 1}$ we denote its {\it grade} by
$$
\partial(x)=\pm1.
$$
Assigning a $\bz_2$-grading on $\ga$ is equivalent to equipping $\ga$ with an involutive $*$-automorphism $\th$ (\emph{i.e.} $\th^2=\id_\ga$). Indeed, from one hand for a given $\bz_2$-graded $*$-algebra $\ga$ one takes
\begin{equation*}
\th\lceil_{\ga_1}=\id_{\ga_1}\,,\quad \th\lceil_{\ga_{-1}}=-\id_{\ga_{-1}}\,.
\end{equation*}
On the other hand, if $\th\in\aut(\ga)$ is such that $\th^2=\id_\ga$, after taking
$$
\eps_1:=\frac{1}{2}(\id_{\ga}+\th)\,, \quad \eps_{-1}:=\frac{1}{2}(\id_{\ga}-\th)\,,
$$
and denoting
$$
\ga_1:=\eps_1(\ga)\,,\quad\ga_{-1}:=\eps_{-1}(\ga)\,,
$$
one gives $\ga_1\cap \ga_{-1}=\{0\}$. As a consequence, their direct sum $\ga=\ga_1\oplus \ga_2$ is a $\bz_2$-graded $*$-algebra.\\ It turns out that a $\bz_2$-graded $*$-algebra is a pair $(\ga,\th)$, where $\ga$ is an involutive $*$-algebra, and $\th$ an involutive $*$-automorphism on $\ga$.\\
Following \cite{CDF}, we say that $\theta$ is a $\bz_2$-grading of $\ga$. Moreover, we denote the $*$-subalgebra $\ga_+:=\ga_1$ the {\it even part}, and the subspace $\ga_-:=\ga_{-1}$ the {\it odd part} of $\ga$, respectively. Note that $\eps_1$ is a conditional expectation.
Thus, for any $a\in\ga$, we can write $a=a_++a_-$, with $a_+\in\ga_+$, $a_-\in\ga_-$,
and this decomposition is unique. In addition, one gets $\th(a_+)=a_+$, $\th(a_-)=-a_-$.\\
Taking $\th=\id_{\ga}$, one sees that any $*$-algebra $\ga$ is equipped with a $\bz_2$ trivial grading. Here, $\ga_+=\ga$ and $\ga_-=\{0\}$.\\
A simple example of $\bz_2$-graded $*$-algebra is obtained by taking an Hilbert space $\ch$, and a bounded self-adjoint unitary $U$ on $\ch$\footnote{As we will see in the sequel, the pair $(\ch, U)$ realizes a $\bz_2$-grading on $\ch$.}. The adjoint action $\ad_{U}(\cdot):=U \cdot U^*$ is an involutive $*$-automorphism which induces a $\bz_2$-grading on $\cb(\ch)$. 
Another example of paramount importance for its applications to Physics is of course the CAR algebra, see {\it e.g.}
\cite{BR2, CDF}.

Let $\big(\ga_i,\th_i\big)$, $i=1,2$, be a pair of $\bz_2$-graded $*$-algebras.
The map $T:\ga_1\to\ga_2$ is said to be {\it even} if it is grading-equivariant:
$$
T\circ\th_1=\th_2\circ T\,.
$$
When $\th_2=\id_{\ga_2}$, the map $T:\ga_1\to\ga_2$ is even if and only if it is grading-invariant, that is
$T\circ\th_1=T$. If $T$ is $\bz_2$-linear,
then it is even if and only if $T\lceil_{\ga_{1,-}}=0$. When $\big(\ga_2,\th_2\big)=\big(\bc,\id_\bc\big)$, a functional $f:\ga_1\to\bc$ is even if and only if $f\circ\th=f$.\\

In the sequel, we will denote by $\cs_+(\ga)$ the convex subset of even states. Even states play a role in giving a $\bz_2$-grading on their GNS structures.\\
More in detail, suppose that $(\ga,\th)$ is a $\bz_2$-graded $C^*$-algebra, and $\f\in\cs_+(\ga)$. Let $(\ch_\f,\pi_\f, \xi_\f, V_{\th,\f})$ be the GNS covariant representation of $\f$, where the unitary self-adjoint $V_{\th,\f}$ fixes $\xi_\f$ and verifies
$$
\pi_\f(\th(a))=V_{\th,\f}\pi_\f(a)V_{\th,\f}\,, \quad a\in\ga\,.
$$
Then, $(\cb(\ch),\ad_{V_{\th,\f}})$ is a $\bz_2$-graded $C^*$-algebra.\\

It is well known that in the case of a $C^*$-algebra with trivial grading, the pure states separate the points and are exactly the extreme elements of the convex of the states. For similar results in presence of a nontrivial $\bz_2$-grading, the notion of grading-invariant functionals is central. To this aim, take a $\bz_2$-graded $C^*$-algebra $(\ga, \theta)$, and denote by
$\ce(\cs_+(\ga))$ the set of the extreme points of the weakly* compact convex set
$\cs_+(\ga)$ of all even states of $\ga$.

\begin{prop}
\label{extremeven}
For any given $\bz_2$-graded $C^*$-algebra $(\ga, \theta)$, the states of $\ce(\cs_+(\ga))$ separate $\ga$.
\end{prop}

\begin{proof}
We first observe that $\cs_+(\ga)$ separate $\ga$, since the conditional expectation $\eps_1$ from $\ga$ to $\ga_+$ is faithful. The statement is then an application of the
Krein-Milman theorem.\\
Indeed, let $a$ be a positive element in $\ga$ such that $\om(a)=0$ for any
$\om\in\ce(\cs_+(\ga))$. Then, for any convex combination $\f$ of states in
$\ce(\cs_+(\ga))$, one has $\varphi(a)=0$. Now, given any even state $\eta$, there exists a net
$\{\eta_i\}_{i\in I}$, $I$ being a directed set, for which $\eta_i$ is a convex combination of states in
$\ce(\cs_+(\ga))$, and $\eta_i\rightarrow_i \eta$ in the weak* topology.
But then we have
$$
\eta(a)=\lim_i \eta_i(a)=0\,.
$$
As the equality
holds for all even states, it follows that $a=0$.
\end{proof}

\begin{cor}
\label{evnorm}
Given a $\bz_2$-graded $C^*$-algebra $(\ga, \theta)$, for every $a\in\ga$ one has
$$
\|a\|=\sup\{\|\pi_\om(a)\| :\om \in\ce(\cs_+(\ga))\}\,.
$$
\end{cor}
\begin{proof}
Thanks to Proposition \ref{extremeven}, the representation $\oplus_{\om\in\ce(\cs_+(\ga))} \pi_\om$ is faithful and thus isometric. This ends the proof.
\end{proof}

\begin{lem}\label{extension}
Let $(\ga, \theta)$ be a $\bz_2$-graded
$C^*$-algebra. Any state $\om$ on the even subalgebra
$\ga_+$ admits exactly one even extension to $\ga$.
\end{lem}

\begin{proof}
Clearly, $\om\circ \eps_1$ extends $\om$ and is even by construction.
The uniqueness of such an extension follows from the fact that even states vanish
on $\ga_-$. 
\end{proof}

\begin{prop}
\label{affine}
Let $(\ga, \theta)$ be a $\bz_2$-graded
$C^*$-algebra. An even state $\om$ belongs to $\ce(\cs_+(\ga))$ if and only if
the restriction $\om\lceil_{\ga_+}$ is a pure state.
\end{prop}
\begin{proof}
In light of Lemma \ref{extension} the map $S: \cs_+(\ga)\rightarrow \cs(\ga_+)$, given by
$S(\om)=\om\lceil_{\ga_+}$ for every $\om\in\cs_+(\ga)$, establishes an affine
bijection between the compact convex sets $\cs_+(\ga)$ and
$\cs(\ga_+)$. Therefore, it also establishes a bijection between the extreme points of the former set
and the extreme points of the latter.
\end{proof}
For commutative $C^*$-algebras, one can completely determine the extreme even states.
\begin{prop}
\label{exevenab}
For a (locally) compact Hausdorff  space $X$, the extreme even states of the commutative $\bz_2$-graded $C^*$-algebra $(C(X),\th)$
are given by the set $\big\{\frac{1}{2}(\delta_x+\delta_{\th(x)})\mid x\in X\big\}$.
\end{prop}
\begin{proof}
Consider the equivalence relation $\sim$ on $X$ for which
$x\sim y$ if and only if $f(x)=f(y)$ for any $f\in C(X)$ such that $f= f\circ\th$. If $X\!/\sim$ is the quotient space, the even $C^*$-subalgebra $C(X)_+:=\{f\in C(X)\mid f= f\circ\theta\}$
clearly identifies with $C(X_+)$. Indeed,
if $\pi: X\rightarrow X_+$ is the natural
projection, and  $\iota: C(X_+)\rightarrow C(X)$
is the isometric $*$-homomorphism given by
$\iota(f)=f\circ\pi$ for any $f\in C(X_+)$, one has that $C(X)_{+}= \iota(C(X_+))$.\\
Thus, for any given $x\in X$, the restriction of the even state
$\frac{1}{2}(\delta_x+\delta_{\th(x)})$ to $C(X)_+$ is
pure, as it is nothing but $\delta_{\pi(x)}$.\\
Conversely, if $\om\in \ce(\cs_+(C(X))$, then there exists
$x\in X$ such that $\om\lceil_{C(X_+)}=\delta_{\pi(x)}$, which means
$\om=\frac{1}{2}(\delta_x+\delta_{\th(x)})$.
\end{proof}

Suppose that $(\ga_1,\th_1)$ and $(\ga_2,\th_2)$ are $\bz_2$-graded $*$-algebras, and consider the linear space $\ga_1\odot \ga_2$. In what follows, we recall the definition of the involutive $\bz_2$-graded tensor product, which will be henceforth denoted by
$\ga_1\, \circled{\text{{\tiny F}}}\, \ga_2$, as in \cite{CDF}. For homogeneous elements $a_1\in\ga_1$, $a_2\in\ga_2$ and $i,j\in\bz_2$, we set
\begin{eqnarray*}
\begin{split}
\label{ecfps}
&\eps(a_1,a_2):=\left\{\!\!\!\begin{array}{ll}
                      -1 &\text{if}\,\, \partial(a_1)=\partial(a_2)=-1\,,\\
                     \,\,\,\,\,1 &\text{otherwise}\,.
                    \end{array}
                    \right.\\
&\eeps(i,j):=\left\{\!\!\!\begin{array}{ll}
                      -1 &\text{if}\,\, i=j=-1\,,\\
                     \,\,\,\,\,1 &\text{otherwise}\,.
                    \end{array}
                    \right.
\end{split}
\end{eqnarray*}
Given $x,y\in\ga_1\odot\ga_2$ with 
\begin{align*}
\begin{split}
\label{xy}
&x:=\oplus_{i,j\in\bz_2}x_{i,j}\in\oplus_{i,j\in\bz_2}(\ga_{1,i}\odot\ga_{2,j})\,,\\
&y:=\oplus_{i,j\in\bz_2}y_{i,j}\in\oplus_{i,j\in\bz_2}(\ga_{1,i}\odot\ga_{2,j})\,,
\end{split}
\end{align*}
 the involution, which by a minor abuse of notation we continue
to denote by $^*$, and  the multiplication on $\ga_1\, \circled{\text{{\tiny F}}}\, \ga_2$ are defined as (see also \emph{e.g.} \cite{CDF}) 
\begin{equation}
\label{prstc}
\begin{split}
x^*:=&\sum_{i,j\in\bz_2}\eeps(i,j)x_{i,j}^*\, ,\\
xy:=&\sum_{i,j,k,l\in\bz_2}\eeps(j,k)x_{i,j}y_{k,l}\,.
\end{split}
\end{equation}
%
The $*$-algebra thus obtained also carries a $\bz_2$-grading, which
is induced by the $*$-automorphism $\th=\th_1\, \circled{\rm{{\tiny F}}}\, \th_2$ given on the elementary tensors by
\begin{equation*}
\label{aupf}
(\th_1\, \circled{\rm{{\tiny F}}}\, \th_2)(a_1\, \circled{\rm{{\tiny F}}}\, a_2):=\th_1(a_1)\, \circled{\rm{{\tiny F}}}\, \th_2(a_2)\,,\quad a_1\in\ga_1\,,\,\, a_2\in\ga_2\,.
\end{equation*}
where $a_1\, \circled{\rm{{\tiny F}}}\, a_2$ is nothing but  $a_1\otimes a_2$ thought of
as an element of the $\bz_2$-graded $*$-algebra $\ga_1\, \circled{\text{{\tiny F}}}\, \ga_2$,
since
$\ga_1\, \circled{\rm{{\tiny F}}}\, \ga_2=\ga_1\odot \ga_2\,$ as linear spaces.
As of now, we will use $a_1\otimes a_2$ and $a_1\, \circled{\rm{{\tiny F}}}\, a_2$ interchangeably 
when no confusion can occur.\\
The even and odd part of the Fermi product are respectively
\begin{equation*}
\label{picm}
\begin{split}
\big(\ga_1\, \circled{\rm{{\tiny F}}}\, \ga_2\big)_+:=&\big(\ga_{1,+}\odot\ga_{2,+}\big)\oplus\big(\ga_{1,-}\odot\ga_{2,-}\big)\,,\\
\big(\ga_1\, \circled{\rm{{\tiny F}}}\, \ga_2\big)_-:=&\big(\ga_{1,+}\odot\ga_{2,-}\big)\oplus\big(\ga_{1,-}\odot\ga_{2,+}\big)\,.
\end{split}
\end{equation*}

\medskip

For $\om_i\in\cs(\ga_i)$, $i=1,2$, the state $\psi_{\om_1,\om_2}$ has a counterpart in $\ga_1\, \circled{\rm{{\tiny F}}}\, \ga_2$ by means of the product functional $\om_1\times \om_2$, defined as usual by
$$
\om_1\times \om_2\bigg(\sum_{j=1}^n a_{1,j}\circled{\rm{{\tiny F}}}\, a_{2,j}\bigg):=\sum_{j=1}^n \om_1(a_{1,j})\om_2(a_{2,j})\,,
$$
for all $\sum_{j=1}^n a_{1,j} \circled{\rm{{\tiny F}}}\, a_{2,j}\in \ga_1\, \circled{\rm{{\tiny F}}}\, \ga_2$. Contrarily to the case of trivial grading, the map defined above is not necessarily positive. The following proposition, which generalizes the results obtained in \cite{AM1} for the CAR algebra, gives a necessary and sufficient condition for the positivity.
\begin{prop}\label{converse}
Let $(\ga_1,\th_1), (\ga_2,\th_2)$ be $\bz_2$-graded $C^*$-algebras, and
$\om_1\in\cs(\ga_1)$, $\om_2\in\cs(\ga_2)$. Then
 $\om_1\times\om_2$ is positive on $\ga_1\,\circled{\rm{{\tiny F}}}\,\ga_2$ if and only if at least one between
$\om_1$ and $\om_2$ is even. Moreover, $\om_1\times\om_2$ is even
if and only if both $\om_1$ and $\om_2$ are even.
\end{prop}
\begin{proof}
The "if" part follows from \cite[Proposition 7.1]{CDF}.
For the "only if" part we shall argue by contradiction, exactly as is done in the proof of \cite[Theorem 1]{AM1}.
Suppose neither $\om_1$ nor $\om_2$ is even. Then there exist odd
$a_1\in\ga_1$ and $a_2\in\ga_2$ such that
$\om_1(a_1)\neq 0$ and $\om_2(a_2)\neq0$. Furthermore, there is no loss
of generality if we also assume that $a_1$ and $a_2$ are both self-adjoint.
Now, on the one hand we have $\om_1\times \om_2\,(a_1\,\circled{\rm{{\tiny F}}}\, a_2)=\om_1(a_1)\om_2(a_2)$ is a real number different from
zero. On the other hand, we  also have
\begin{align*}
\overline{\om_1\times\om_2}\, (a_1\,\circled{\rm{{\tiny F}}}\, a_2)&=\om_1\times\om_2\, ((a_1\,\circled{\rm{{\tiny F}}}\, a_2)^*)\\
&=-\om_1\times\om_2(a_1^*\otimes a_2^*)\\
&=-\om_1\times\om_2(a_1\,\circled{\rm{{\tiny F}}}\, a_2)
\end{align*}
from which we see that  $\om_1\times\om_2\, (a_1\,\circled{\rm{{\tiny F}}}\, a_2)$ must be zero, and a contradiction is thus arrived at.\\
If $\om_1$ and $\om_2$ are both even, then the product is seen at once to be even as well. Conversely if
$\om_1\times\om_2$ is even, then $\om_1$ and $\om_2$ must be both even. For instance, for $\om_1$ we have
$$
\om_1(\theta_1(a_1))=\om_1\times\om_2\, (\theta_1\,\circled{\rm{{\tiny F}}}\, \theta_2 (a_1\otimes \idd_{\ga_2}))=\om_1\times\om_2\, (a_1\otimes \idd_{\ga_2})=\om_1(a_1)\,,
$$
for every $a_1\in\ga_1$.
\end{proof}
We would like to remark that the above proposition 
holds for products of an arbitrary number $n\geq 2$ of states $\om_1, \om_2, \ldots, \om_n$. More precisely,
the product functional $\om_1\times\om_2\times\ldots\times\om_n$
will be positive if and only if $n-1$ of the marginal states are even, and it will be
even if and only if all the marginal states are so.

\section{Fermi tensor product of Hilbert spaces}\label{hilbert}

A $\bz_2$-graded Hilbert space is a pair $(\ch, U)$, where $\ch$ is
a (complex) Hilbert space and $U$ a self-adjoint unitary acting on $\ch$.\\
Note that $\ch$ decomposes into a direct sum
$$\ch=\ch_+\oplus\ch_-$$
 where $\ch_+:={\rm Ker}(I-U)$,
$\ch_-:={\rm Ker}(I+U)$, and $I$ is the identity operator.
As usual, vectors belonging to $\ch_+$ ($\ch_-$) are referred to as {\it even} ({\it odd}) vectors, and elements belonging to any of these subspaces are collectively referred to as
homogeneous vectors. The grade $\partial(\xi)$ of any homogeneous vector $\xi$ is $1$ or $-1$, according to whether it belongs to $\ch_+$ or $\ch_-$, respectively. Recalling that $(\cb(\ch),\ad_U)$ is a $\bz_2$-graded $*$-algebra, after applying a homogeneous element of $\cb(\ch)$ to a homogeneous vector in $\ch$ we get a homogenous vector whose grade can be easily determined, as the following result shows.
\begin{lem}
Let $(\ch, U)$ be a $\bz_2$-graded Hilbert space. If $T\in \cb({\ch})$ and
$\xi\in\ch$ are both homogeneous, then $T\xi$ is still homogeneous and
$$\partial(T\xi)=\partial(T)\partial(\xi).$$
\end{lem}

\begin{proof}
There are four cases to deal with, which can be treated in the same way.
For instance, if $T$ is even and $\xi$ is odd, then we have
$UT\xi=(UTU^*)U\xi=T(-\xi)=-T\xi$, that is $T\xi$ is odd.
\end{proof}
Let $(\ga, \theta)$ be a $\bz_2$-graded $C^*$-algebra, and let $\om$ be an even state of
$\ga$, whose covariant GNS representation is
$(\ch_\om, \pi_\om, \xi_\om, V_{\theta, \om})$. Thus, $(\ch_\om, V_{\theta, \om})$ is a $\bz_2$-graded Hilbert space.
The following result allows us to realize the even part of $\ch_\om$.
\begin{prop}
The even part of $\ch_\om$ is given by $\ch_{\om, +}=\overline{\pi_\om(\ga_+)\xi_\om}$
\end{prop}
\begin{proof}
From the obvious inclusion ${\pi_\om(\ga_+)\xi_\om}\subset\ch_{\om, +}$, one also has
$\overline{{\pi_\om(\ga_+)\xi_\om}}\subset\ch_{\om, +}$.\\
For the converse inclusion, pick $\eta\in\ch_{\om, +}$, that is $V_{\theta, \om}\eta=\eta$.
By cyclicity of $\xi_\om$, there exists a sequence $\{a_n\}_{n\in\bn}$
such that $\eta=\lim_{n\rightarrow\infty}\pi_\om(a_n)\xi_\om$. But then
$V_{\theta, \om}\eta=V_{\theta, \om}\big(\lim_{n\rightarrow\infty}\pi_\om(a_n)\xi_\om\big)=\lim_{n\rightarrow\infty}\pi_\om(\theta(a_n))\xi_\om$.
From the equality $\eta=\frac{1}{2}(\eta+V_{\theta, \om}\eta)$ we see that $\eta=\lim_{n\rightarrow\infty}\pi_\om(\varepsilon_1(a_n))\xi_\om$, and
the thesis follows since $\varepsilon_1(a_n)\in \ga_+$.
\end{proof}

The tensor product $\ch_1\otimes\ch_2$ of
 two $\bz_2$-graded Hilbert spaces $(\ch_1, U_1)$ and $(\ch_2, U_2)$
can be endowed  with  the natural grading induced by the self-adjoint
unitary $U_1\otimes U_2$. Note that a simple tensor
$\xi_1\otimes\xi_2$ in $\ch_1\otimes\ch_2$ is even precisely when 
$\xi_1$ and $\xi_2$ are both even or when they are both odd.
In the sequel, most of the times the Hilbert space $\ch_1\otimes\ch_2$ will be
thought of as a graded Hilbert space, with grading $U_1\otimes U_2$.
Rather than write $(\ch_1\otimes\ch_2, U_1\otimes U_2 )$, we will simply denote this graded Hilbert space
by $\ch_1\,\circled{\rm{{\tiny F}}}\,\ch_2$.\\
Given  $T_1\in \cb(\ch_1)$ and a homogeneous operator $T_2\in\cb(\ch_2)$,
we  define $T_1\odot T_2$ as the linear operator
acting on $\ch_1\odot\ch_2$ as
\begin{equation}\label{fermiprodop}
T_1\odot T_2 (\xi \odot \eta)=\varepsilon (T_2, \xi )(T_1 \xi\odot T_2 \eta)
\end{equation}
for any homogeneous $\xi\in\ch_1$ and any $\eta\in\ch_2$.
Clearly, by using \eqref{fermiprodop} one can extend $T_1\odot T_2$ to the whole vector space
$\ch_1\odot\ch_2$. Furthermore, it is easy to see that $T_1\odot T_2$ is a bounded operator, and thus it can be extended by continuity to
$\ch_1\circled{\rm{{\tiny F}}}\,\ch_2$. Its unique extension is denoted by $T_1\,\circled{\rm{{\tiny F}}}\, T_2$.\\

Let now $(\ga_1, \theta_1), (\ga_2, \theta_2)$  be $\bz_2$-graded
$C^*$-algebras. If $\pi_i: \ga_i\rightarrow\cb(\ch_i)$ are grading-equivariant representations, it is possible to
define a map,
$\pi_1\,\circled{\rm{{\tiny F}}}\, \pi_2$, of $\ga_1\,\circled{\rm{{\tiny F}}}\, \ga_2$ acting on the
Fermi tensor product $\ch_1\,\circled{\rm{{\tiny F}}}\, \ch_2$ as
$$\pi_1\,\circled{\rm{{\tiny F}}}\, \pi_2 (a_1\circled{\rm{{\tiny F}}}\, a_2):=\pi_1(a_1)\,\circled{\rm{{\tiny F}}}\, \pi_2(a_2)$$
for every $a_1\in\ga_1$ and $a_2\in\ga_2$.
\begin{prop}
Under the above assumptions, one has that $\pi_1\,\circled{\rm{{\tiny F}}}\, \pi_2$ is a $*$-representation
of $\ga_1\,\circled{\rm{{\tiny F}}}\, \ga_2$ acting on the
Fermi tensor product $\ch_1\,\circled{\rm{{\tiny F}}}\, \ch_2$ .
\end{prop}
\begin{proof}
All we have to do is ascertain that $\pi_1\,\circled{\rm{{\tiny F}}}\, \pi_2$
preserves both product and involution of $\ga_1\,\circled{\rm{{\tiny F}}}\, \ga_2$.
To this end, it is enough to verify the involved equalities only on homogeneous elements.\\
Let $a_i\in\ga_i$ be homogeneous elements, $i=1, 2$, $\xi, \xi'\in\ch_1$
$\eta, \eta' \in \ch_2$ be homogeneous vectors, and observe that $\p_i(a_i)$ has the same grade of $a_i$, since $\pi_i$ is grading-equivariant. \\
On the one hand, from \eqref{prstc} and \eqref{fermiprodop}, we have
\begin{align*}
&\langle\pi_1\,\circled{\rm{{\tiny F}}}\, \pi_2\big((a_1\otimes a_2)^*\big) (\xi\otimes\eta), \xi'\otimes\eta' \rangle\\
=&\varepsilon(a_1, a_2)\langle \pi_1(a_1^*)\,\circled{\rm{{\tiny F}}}\,\pi_2(a_2^*)\, (\xi\otimes\eta), \xi'\otimes\eta'  \rangle\\
=&\varepsilon(a_1, a_2)\varepsilon(a_2, \xi) \langle \pi_1(a_1^*)\xi\otimes\pi_2(a_2^*)\eta, \xi'\otimes\eta'\rangle\\
=&\varepsilon(a_1, a_2)\varepsilon(a_2, \xi)
\langle \xi, \pi_1(a_1) \xi' \rangle \langle \eta, \pi_2(a_2) \eta' \rangle\,.
\end{align*}
On the other hand, from \eqref{fermiprodop} we have
\begin{align*}
&\langle\big(\pi_1\,\circled{\rm{{\tiny F}}}\, \pi_2(a_1\otimes a_2)\big)^* (\xi\otimes\eta), \xi'\otimes\eta' \rangle\\
=&\langle \xi\otimes\eta,\pi_1\,\circled{\rm{{\tiny F}}}\, \pi_2(a_1\otimes a_2) (\xi'\otimes\eta' )\rangle\\
=&\varepsilon(a_2, \xi') \langle \xi\otimes\eta, \pi_1(a_1)\xi'\otimes\pi_2(a_2)\eta'\rangle\\
=&\varepsilon(a_2, \xi')\langle \xi, \pi_1(a_1) \xi' \rangle \langle \eta, \pi_2(a_2) \eta' \rangle\,,
\end{align*}
and the two expressions equal each other because 
$$\varepsilon(a_1, a_2)\varepsilon(a_2, \xi)=\varepsilon(a_2, \xi')$$
whenever $\langle \xi, \pi_1(a_1) \xi' \rangle$ is different from $0$,
as a painstaking inspection of the signs shows.\\

\noindent As for the product, pick homogeneous $a_1, b_1\in\ga_1$ and $a_2, b_2\in\ga_2$.
With homogeneous $\xi\in\ch_1, \eta\in\ch_2$ by applying \eqref{prstc} and \eqref{fermiprodop} we have
\begin{align*}
&\pi_1\,\circled{\rm{{\tiny F}}}\, \pi_2 \big((a_1\otimes a_2)(b_1\otimes b_2)\big)(\xi\otimes\eta)\\
=&\varepsilon(a_2, b_1)\pi_1\,\circled{\rm{{\tiny F}}}\, \pi_2(a_1b_1\otimes a_2b_2)(\xi\otimes\eta)\\
=&\varepsilon(a_2, b_1)\varepsilon(a_2 b_2, \xi)\pi_1(a_1b_1)\xi\otimes\pi_2(a_2b_2)\eta\\
=&\varepsilon(a_2, b_1)\varepsilon(a_2 b_2, \xi)\pi_1(a_1)\pi_1(b_1)\xi\otimes\pi_2(a_2)\pi_2(b_2)\eta\,.
\end{align*}
On the other hand, by applying  \eqref{fermiprodop} we have
\begin{align*}
&\pi_1\,\circled{\rm{{\tiny F}}}\, \pi_2 (a_1\otimes a_2)[ \pi_1\,\circled{\rm{{\tiny F}}}\, \pi_2 (b_1\otimes b_2)(\xi\otimes\eta)]\\
=&\varepsilon(b_2, \xi)\pi_1\,\circled{\rm{{\tiny F}}}\, \pi_2 (a_1\otimes a_2) \big(\pi_1(b_1)\xi\otimes\pi_2(b_2)\eta\big)\\
=&\varepsilon(b_2, \xi)\varepsilon(a_2, \pi_1(b_1)\xi)\pi_1(a_1)\pi_1(b_1)\xi\otimes \pi_2(a_2)\pi_2(b_2)\eta\,.
\end{align*}
Again, a painstaking inspection of the signs shows that
$$\varepsilon(a_2, b_1)\varepsilon(a_2 b_2, \xi)=\varepsilon(b_2, \xi)\varepsilon(a_2, \pi_1(b_1)\xi)$$
and the proof is complete.
\end{proof}
The following result says that $\pi_{\om_1}\,\circled{\rm{{\tiny F}}}\,\pi_{\om_2}$ is the GNS representation of the product state of two even states $\om_1$ and $\om_2$.
\begin{prop}
\label{GNSprod}
Let $(\ga_i, \theta_i)$ be $\bz_2$-graded $C^*$-algebras, $i=1, 2$.
If $\om_i\in\mathcal{S}(\mathcal{\ga}_i)$ are even states for $i=1, 2$ then
$$\pi_{\om_1\times\om_2}=\pi_{\om_1}\,\circled{\rm{{\tiny F}}}\,\pi_{\om_2}$$
up to unitary equivalence, as representations of the Fermi tensor product
$\ga_1\,\circled{\rm{{\tiny F}}}\,\ga_2$.
\end{prop}
\begin{proof}
First note that $\pi_{\om_1}\,\circled{\rm{{\tiny F}}}\,\pi_{\om_2}$ is a cyclic representation
of $\ga_1\,\circled{\rm{{\tiny F}}}\,\ga_2$ with cyclic vector
$\xi_{\om_1}\otimes\xi_{\om_2}$,  where
$\xi_{\om_i}$ is the GNS vector of $\om_i$ for $i=1, 2$. Therefore, to conclude it is enough to make sure the vector
state associated with $\xi_{\om_1}\otimes\xi_{\om_2}$ coincides with the product
state $\om_1\times\om_2$. To this aim, consider homogeneous $a_i\in\ga_i$, $i=1, 2$. After recalling that the cyclic vectors are even, by applying \eqref{fermiprodop} one finds
\begin{align*}
&\langle \pi_{\om_1}\,\circled{\rm{{\tiny F}}}\,\pi_{\om_2}(a_1\otimes a_2) \xi_{\om_1}\otimes\xi_{\om_2} , \xi_{\om_1}\otimes\xi_{\om_2}\rangle\\
=&\varepsilon(a_2, \xi_{\om_1})\langle \pi_{\om_1}(a_1)\xi_{\om_1}\otimes\pi_{\om_2}(a_2)\xi_{\om_2}, \xi_{\om_1}\otimes\xi_{\om_2} \rangle\\
=&
\langle\pi_{\om_1}(a_1)\xi_{\om_1}, \xi_{\om_1}\rangle \langle\pi_{\om_2}(a_2)\xi_{\om_2}, \xi_{\om_2}\rangle\\
=&\om_1(a_1)\om_2(a_2)=\om_1\times\om_2\,(a_1\otimes a_2)
\end{align*}
which ends the proof.
\end{proof}

\section{Norms on $\bz_2$-graded tensor product of $C^*$-algebras}\label{spnorm}
In this section we  undertake a detailed study of the so-called spatial norm on the Fermi tensor product of $\bz_2$-graded $C^*$-algebras. 
We show that it is minimal besides being a cross norm, as is the maximal one already introduced in \cite{CDF}. 
Given $\bz_2$-graded $C^*$-algebras $(\ga_1,\th_1)$ and $(\ga_2,\th_2)$, the latter norm is given  by
$$
\|x\|_{\max}:=\sup\{\|\pi(x)\|: \pi\,\,\text{is a representation}\}\,,
$$
for all $x\in\ga_1\, \circled{\rm{{\tiny F}}}\,\ga_2$, and it
is obviously the biggest norm on $\ga_1\, \circled{\rm{{\tiny F}}}\,\ga_2$.\\
The spatial norm is defined in terms of the GNS representations of products
of even states. More precisely,
 
$$
\|x\|_{\min}:=\sup \{\|\pi_{\om_1\times\om_2}(x)\|: \om_1\in\cs_+(\ga_1)\,,\,\, \om_2\in\cs_+(\ga_2)\}\,,
$$
for all $x\in\ga_1\,\circled{\rm{{\tiny F}}}\,\ga_2$.
In principle, the above definition might provide only a seminorm. 
In fact, it is known that it actually defines a norm. For want of a reference, we nonetheless include a full proof of this fact. 
\begin{prop}
Under the above assumptions,  the seminorm $\|\cdot\|_{\min}$ is a
$C^*$-norm on the Fermi product $\ga_1\,\circled{\rm{{\tiny F}}}\,\ga_2$.
\end{prop}
\begin{proof}
As the seminorm $\|\cdot\|_{\min}$ obviously satisfies the $C^*$-equality, we only need to make sure it also
separates the elements of  $\ga_1\,\circled{\rm{{\tiny F}}}\,\ga_2$.
Let $c=\sum_{i=1}^n a_i\otimes b_i $ be such an element with $c\neq 0$,
where $a_i\in\ga_1$ and $b_i\in\ga_2$, $i=1, 2, \ldots, n$.
Take now
$$
\widetilde{\ga_1}:= C^*(\{a_i, \theta_1(a_i)\mid i=1, 2, \ldots, n\})
$$
and
$$
\widetilde{\ga_2}:= C^*(\{b_i, \theta_2(b_i)\mid i=1, 2, \ldots, n\} )\,.
$$
By definition, $\widetilde{\ga_1}$ and $\widetilde{\ga_2}$ inherit
the $\bz_2$-grading from $\ga_1$ and  $\ga_2$, respectively.
Since $\widetilde{\ga_1}$ and $\widetilde{\ga_2}$ are separable
$C^*$-algebras, there exist faithful states $\widetilde{\om_1}$ and $\widetilde{\om_2}$ on
$\widetilde{\ga_1}$ and $\widetilde{\ga_2}$, respectively.
If $\widetilde{\varepsilon}_i$ is the canonical conditional expectation from
$\widetilde{\ga_i}$ onto $\widetilde{\ga}_{i,+}$, $i=1,2$,
then  $\widetilde{\om_i}\circ\widetilde{\varepsilon_i}$ are still faithful. But then by \cite[Theorem IV.4.9 (iii)]{T1}, the product state
$\psi_{\widetilde{\om_1}\circ\widetilde{\varepsilon_1},\widetilde{\om_2}\circ\widetilde{\varepsilon_2}}$
is faithful on $\widetilde{\ga_1}\otimes_{\rm min}\widetilde{\ga_2}$. Therefore,
we have
\begin{equation}
\label{faith}
\widetilde{\om_1}\circ\widetilde{\varepsilon_1}\times\widetilde{\om_2}\circ\widetilde{\varepsilon_2}\, (c^*c)=\psi_{\widetilde{\om_1}\circ\widetilde{\varepsilon_1},\widetilde{\om_2}\circ\widetilde{\varepsilon_2}}(c^* c)>0\,.
\end{equation}
The conclusion now readily
follows by considering any extensions $\om_1, \om_2$ of $\widetilde{\om_1}, \widetilde{\om_2}$ to the whole
$\ga_1$ and $\ga_2$, respectively. Indeed, from \eqref{faith}
$$
\|\pi_{\om_1\circ\eps_1\times \om_2\circ\eps_2}(c)\|^2=\om_1\circ\varepsilon_1 \times\om_2\circ\varepsilon_2\,(c^*c)>0\,.
$$
\end{proof}
Here, we present the definition of  a $C^*$-cross norm in the presence of a $\bz_2$-grading, which reduces to the usual one when the grading is trivial.
\begin{defin}
For any given $\bz_2$-graded $C^*$-algebras $(\ga_1,\th_1)$ and $(\ga_2,\th_2)$, a $C^*$-norm $\|\cdot\|_{\b}$ on $\ga_1\,\circled{\rm{{\tiny F}}}\,\ga_2$ is said to be cross if
$$
\|a_1\circled{\rm{{\tiny F}}}\,a_2\|_{\b}=\|a_1\| \|a_2\|\,, \quad \text{for homogeneous $a_1\in\ga_1\,, a_2\in\ga_2$}\,.
$$
\end{defin}
As one could expect, the norms introduced above are cross:
\begin{prop}\label{cross}
Under the above assumptions, both $\|\cdot\|_{\rm max}$ and $\|\cdot\|_{\min}$ are
cross norms on  $\ga_1\,\circled{\rm{{\tiny F}}}\,\ga_2$.
\end{prop}
\begin{proof}
We start by observing that if $\|\cdot\|_\beta$ is any $C^*$-norm on the Fermi
product $\ga_1\,\circled{\rm{{\tiny F}}}\,\ga_2$, then
$\|a\otimes b\|_\beta\leq \|a\|\|b\|$ for any $a\in\ga_1$, $b\in\ga_2$. Indeed, 
one has
$$
\|a\otimes b\|_\beta=\|(a\otimes \idd_{\ga_2})(\idd_{\ga_1}\otimes b)\|_\beta\leq \|a\otimes \idd_{\ga_2}\|_\beta \|\idd_{\ga_1}\otimes b\|_\beta=\|a\|\|b\|
$$
because both
$$
\ga_1\ni a\mapsto a\otimes \idd_{\ga_2}\in\ga_1\,\circled{\rm{{\tiny F}}}\,\ga_2\,, \quad \ga_2\ni b\mapsto \idd_{\ga_1}\otimes b\in\ga_1\,\circled{\rm{{\tiny F}}}\,\ga_2
$$
are injective $^*$-homomorphisms.\\
Since $\|a\otimes b\|_{\rm max}\geq \|a\otimes b\|_{\min}$, for any $a\in\ga_1$ and $b\in\ga_2$,
it is enough to prove that
$\|a\otimes b\|_{\rm min}\geq \|a\|\|b\|$ for homogeneous $a,b$.
To this aim, note that a straightforward application of Proposition  \ref{GNSprod}
gives $\|a\otimes b\|_{\min}\geq \|\pi_{\om_1}(a)\|\,\|\pi_{\om_2}(b)\|$ for any
even states $\om_i$ on $\cs_+(\ga_i)$, $i=1,2$. The thesis then follows from Corollary \ref{evnorm}.
\end{proof}

Henceforth the completion of $\ga_1\,\circled{\rm{{\tiny F}}}\,\ga_2$ with respect to the spatial norm
$\|\cdot\|_{\min}$
will always be denoted by $\ga_1\,\circled{\rm{{\tiny F}}}_{\min}\,\ga_2$, whereas the completion
with respect to the maximal norm $\|\cdot\|_{\rm max}$ will be denoted by  $\ga_1\,\circled{\rm{{\tiny F}}}_{\rm max}\,\ga_2$.
%
%
We next show that the tensor product of faithful GNS representations of even states is still a faithful representation on the $C^*$-algebra $\ga_1\,\circled{\rm{{\tiny F}}}_{\min}\,\ga_2$. 
To this aim, we first need a technical lemma on normal states.
We say that a state $\om$ of a given $C^*$-algebra $\ga$ is normal in a representation 
$\pi:\ga\rightarrow\cb(\ch)$ if $\om(a)={\rm Tr}(\pi(a)T),\,\,a\in\ga$, for 
some positive trace-class operator $T$ with ${\rm Tr}(T)=1$.

\begin{lem}
\label{tracialstates}
Let us take $i=1,2$, and $\om_i\in\mathcal{S}_+(\ga_i)$. Suppose  for each $i=1, 2$ we are given
a normal  state in $\pi_{\om_i}$, say
$\varphi_i(\cdot):={\rm Tr} (\pi_{\om_i}(\cdot)T_i)$, where $T_i\in\cb(\ch_{\om_i})$ is a positive trace-class operator with unit trace. Then the product state
$\varphi_1\times\varphi_2$ on $\ga_1\,\circled{\rm{{\tiny F}}}\,\ga_2$
is a normal state in $\pi_{\om_1\times\om_2}$. More precisely, one has
$$
\varphi_1\times\varphi_2(\cdot)= {\rm Tr}(\pi_{\om_1}\,\circled{\rm{{\tiny F}}}\,\pi_{\om_2}(\cdot)\,\, T_1\,\circled{\rm{{\tiny F}}}\,T_2)\,.
$$
\end{lem}

\begin{proof}
We first point out that no loss of generality occurs if we suppose  that each $T_i$ is even for $i=1, 2$. Indeed, if $T_i$ is not even, and $(\ch_{\om_1}, \pi_{\om_i}, \xi_{\om_i}, V_{\th_i,\om_i})$ is the GNS covariant representation of $\om_i$, then it is enough
to pass to $T_i':= \frac{1}{2}(T_i+V_{\th_i,\om_i}TV_{\th_i,\om_i})$, since one can check that $\f_i(\cdot)={\rm Tr}(\pi_{\om_i}(\cdot)\, T_i')$. \\
With homogeneous $a_i\in\ga_i$ and homogeneous vectors $\xi_i\in\ch_{\om_i}$, by using 
\eqref{prstc} and \eqref{fermiprodop} one can easily see that
\begin{equation}
\label{prodf}
\big(\pi_{\om_1}\,\circled{\rm{{\tiny F}}}\,\pi_{\om_2}(a_1\otimes a_2)\big)\,\, T_1\,\circled{\rm{{\tiny F}}}\,T_2\,
(\xi_1\otimes\xi_2)=\varepsilon(a_2, T_1\xi_1)\pi_{\om_1}(a_1)T_1\xi_1\otimes\pi_{\om_2}(a_2)T_2\xi_2
\end{equation}
In order to compute the trace of $A:=\big(\pi_{\om_1}\,\circled{\rm{{\tiny F}}}\,\pi_{\om_2}(a_1\otimes a_2)\big)\,\, T_1\,\circled{\rm{{\tiny F}}}\,T_2$, we need to single out a convenient orthonormal basis
of the Hilbert space $\ch_{\om_1}\circled{\rm{{\tiny F}}}\,\ch_{\om_2}$. To this end, note that
$T_1$ and $T_2$ preserve the homogeneous subspaces of $\ch_{\om_1}$ and
$\ch_{\om_2}$, respectively.
Therefore, one can exhibit a orthonormal basis of $\ch_{\om_i}$
made up of homogeneous eigenvectors of $T_i$, $i=1, 2$.
If we take $I,J,L,K$ as set of indices, we denote by $\{e_i\}_{i\in I}\cup\{e'_j\}_{j\in J}$ the basis of the eigenvectors of $T_1$, where the $e_i$'s are even and the
$e'_j$'s are odd. Likewise, we denote by $\{f_l\}_{l\in L}\cup\{f'_k\}_{k\in K}$ the basis of the eigenvectors of $T_2$, where the $f_l$'s are even and the
$f'_k$'s are odd. As is easily verified, the set
$$\{e_i\otimes f_l\}_{i\in I, l\in L}\cup\{e_i\otimes f'_k\}_{ i\in I, k\in K}\cup\{e'_j\otimes f_l\}_{j\in J, l\in L}\cup\{e'_j\otimes f'_k\}_{ j\in J, k\in K}$$
 is an orthonormal basis of the Fermi
Hilbert space $\ch_{\om_1}\,\circled{\rm{{\tiny F}}}\, \ch_{\om_2}$. Consequently, by \eqref{prodf} and taking into account the orthogonality relations one has:\\
\begin{align*}
Tr (A)= &\sum_{i\in I, l\in L} \langle A\,e_i\otimes f_l, e_i\otimes f_l\rangle\\
&+\sum_{i\in I, k\in K} \langle A\,e_i\otimes f'_k, e_i\otimes f'_k\rangle\\
&+ \sum_{j\in J, l\in L} \langle A\,e'_j\otimes f_l, e'_j\otimes f_l\rangle\\
&+\sum_{j\in J, k\in K} \langle A\,e'_j\otimes f'_k, e'_j\otimes f'_k\rangle\\
=&\sum_{i\in I, l\in L}\langle \pi_{\om_1}(a_1)T_1e_i, e_i \rangle \langle \pi_{\om_2}(a_2)T_2f_l, f_l \rangle \\
&+\sum_{i\in I, k\in K}\langle \pi_{\om_1}(a_1)T_1e_i, e_i \rangle \langle \pi_{\om_2}(a_2)T_2f'_k, f'_k \rangle \\
&+\sum_{j\in J, l\in L}\langle \pi_{\om_1}(a_1)T_1e'_j, e'_j \rangle \langle \pi_{\om_2}(a_2)T_2f_l, f_l \rangle \\
&+\sum_{j\in J, k\in K}\langle \pi_{\om_1}(a_1)T_1e'_j, e'_j \rangle \langle \pi_{\om_2}(a_2)T_2f'_k, f'_k \rangle \\
=& Tr(\pi_{\om_1}(a_1)T_1)Tr(\pi_{\om_2}(a_2)T_2)\\
=& \varphi_1\times\varphi_2(a_1\otimes a_2)\,.
\end{align*}
\end{proof}

\begin{prop}
Let us take $i=1,2$, $(\ga_i,\th_i)$ a $\bz_2$-graded $C^*$-algebra, and $\om_i\in\cs_+(\ga_i)$. If $\pi_{\om_i}$ is a faithful representation, then $\pi_{\om_1}\,\circled{\rm{{\tiny F}}}\,\pi_{\om_2}$ is a faithful representation of $\ga_1 \,\circled{\rm{{\tiny F}}}_{\min}\,\ga_2$
\end{prop}
\begin{proof}
First observe that $\pi_{\om_1}\,\circled{\rm{{\tiny F}}}\,\pi_{\om_2}$ can be extended to
a representation of $\ga_1 \,\circled{\rm{{\tiny F}}}_{\min}\,\ga_2$. Indeed, by definition of
$\|\cdot\|_{\min}$ one has $\|\pi_{\om_1}\,\circled{\rm{{\tiny F}}}\,\pi_{\om_2}\, (x)\|\leq \|x\|_{\min}$
for every $x\in\ga_1 \,\circled{\rm{{\tiny F}}}\,\ga_2$.\\
Therefore,  by density we  need only prove  that $\|x\|_{\min}\leq\|\pi_{\om_1}\,\circled{\rm{{\tiny F}}}\,\pi_{\om_2}\, (x)\|$ for every $x\in\ga_1 \,\circled{\rm{{\tiny F}}}\,\ga_2$.

Let $\varphi_i$ be even state on $\ga_i$, $i=1, 2$. By cyclicity we have
\begin{eqnarray*}
&& \|\pi_{\varphi_1\times\varphi_2}\,(x)\|=\sup_{y:\,\varphi_1\times\varphi_2\,(y^*y)\neq 0}\frac{ \varphi_1\times\varphi_2\,( y^*x^*xy)^\frac{1}{2}}{\varphi_1\times\varphi_2\,(y^*y)^\frac{1}{2}}\,.
\end{eqnarray*}
Since $\pi_{\om_i}$ is faithful, $\varphi_i$ is a weak* limit of a net of states that are normal in the representation
$\pi_{\om_i}$, for $i=1, 2$, {\it cf.} \cite[Theorem 4.9 (iii)]{T1}. In other terms, there exist two nets $\{T_k\}_{k\in K}\subset\mathcal{B}(\ch_{\om_1})$ and $\{S_k \}_{k\in K}\subset\mathcal{B}(\ch_{\om_2})$, of normalized positive trace-class operators such that for every $a_i\in\ga_i$ one has
$$
\varphi_1(a_1)=\lim_k Tr (\pi_{\om_1}(a_1)T_k)\,, \quad \varphi_2(a_2)=\lim_k Tr(\pi_{\om_2}(a_2)S_k)\,.
$$
By virtue of Lemma \ref{tracialstates}, the product state $\varphi_1\times\varphi_2$ is seen at once
to be the weak* limit of the net $\{\eta_k\}_{k\in K}$ with
$$
\eta_k(x):=Tr\big(\pi_{\om_1}\,\circled{\rm{{\tiny F}}}\,\pi_{\om_2}\,(x) T_k\circled{\rm{{\tiny F}}}\, S_k\big)\,, \quad x\in\ga_1 \,\circled{\rm{{\tiny F}}}\,\ga_2\,.
$$
Now, fix $\varepsilon>0$ and let $\widetilde{y}\in\ga_1 \,\circled{\rm{{\tiny F}}}\,\ga_2$ with
$\varphi_1\times\varphi_2\, (\widetilde{y}^*\widetilde{y})\neq 0$ be such that:
$$
\sup_{y:\, \varphi_1\times\varphi_2(y^*y)\neq 0}\frac{ \varphi_1\times\varphi_2( y^*x^*xy)^\frac{1}{2}}{\varphi_1\times\varphi_2(y^*y)^\frac{1}{2}}\leq \frac{ \varphi_1\times\varphi_2( \widetilde{y}^*x^*x\widetilde{y})^\frac{1}{2}}{\varphi_1\times\varphi_2(\widetilde{y}^*\widetilde{y})^\frac{1}{2}}+\frac{\varepsilon}{2}.
$$
Let $k_o\in K$ be such that
$$\frac{ \varphi_1\times\varphi_2( \widetilde{y}^*x^*x\widetilde{y})^\frac{1}{2}}{\varphi_1\times\varphi_2(\widetilde{y}^*\widetilde{y})^\frac{1}{2}}\leq \frac{\eta_{k_o}( \widetilde{y}^*x^*x\widetilde{y})^\frac{1}{2}}{\eta_{k_o}(\widetilde{y}^*\widetilde{y})^\frac{1}{2}}+\frac{\varepsilon}{2}.
$$
We have
\begin{align*}
\sup_{y:\, \varphi_1\times\varphi_2(y^*y)\neq 0}\frac{ \varphi_1\times\varphi_2( y^*x^*xy)^\frac{1}{2}}{\varphi_1\times\varphi_2(y^*y)^\frac{1}{2}}&\leq \frac{\eta_{k_o}( \widetilde{y}^*x^*x\widetilde{y})^\frac{1}{2}}{\eta_{k_o}(\widetilde{y}^*\widetilde{y})^\frac{1}{2}}+\varepsilon\\
&\leq \|\pi_{\om_1}\,\circled{\rm{{\tiny F}}}\,\pi_{\om_2}\,(x) \| +\varepsilon\,.
\end{align*}
Since $\varepsilon>0$ is arbitrary, we find that
$$
\sup_{y:\, \varphi_1\times\varphi_2(y^*y)\neq 0}\frac{ \varphi_1\times\varphi_2( y^*x^*xy)^\frac{1}{2}}{\varphi_1\times\varphi_2(y^*y)^\frac{1}{2}}\leq \|\pi_{\om_1}\,\circled{\rm{{\tiny F}}}\,\pi_{\om_2}\,(x) \|
$$
for any pair of even states $\varphi_i$, $i=1, 2$. Therefore, we finally see that
 $\|x\|_{\min}\leq \|\pi_{\om_1}\,\circled{\rm{{\tiny F}}}\,\pi_{\om_2}\,(x)\|$, and the proof
is complete.
\end{proof}

In the next proposition, we will see that for both norms introduced, the grading on $\ga_1 \circled{\rm{{\tiny F}}} \ga_2$  extends to the $C^*$-completion. Whenever this happens, we call the norm compatible.
\begin{defin}
A $C^*$-norm $\|\cdot\|_\b$ on $\ga_1\,\circled{\rm{{\tiny F}}}\,\ga_2$ is said to be \emph{compatible} if the natural
grading $\theta_1\,\circled{\rm{{\tiny F}}}\,\theta_2$ extends to a (necessarily involutive)
$*$-automorphism of the completion $\ga_1\,\circled{\rm{{\tiny F}}}_\beta\,\ga_2$.
\end{defin}

\begin{prop}
The maximal norm $\|\cdot\|_{\rm max}$ and the spatial norm
$\|\cdot\|_{\min}$ are compatible.
\end{prop}

\begin{proof}
We have to show that $\theta:=\theta_1\,\circled{\rm{{\tiny F}}}\, \theta_2$
can be extended to a $*$-automorphism of both $\ga_1\,\circled{\rm{{\tiny F}}}_{\rm max}\,\ga_2$
and $\ga_1\,\circled{\rm{{\tiny F}}}_{\min}\,\ga_2$.

As for the maximal norm, for $x\in\ga_1\,\circled{\rm{{\tiny F}}}\,\ga_2$ we have
\begin{align*}
\|\theta(x)\|_{\rm max}&=\sup\{\|\pi(\theta(x))\|: \pi\,\textrm{ is a}\, *\textrm{-representation of}\,\ga_1\,\circled{\rm{{\tiny F}}}\,\ga_2 \}\\
&=\sup\{\|\pi'(x) \|: \pi'\,\textrm{ is a}\, *\textrm{-representation of}\,\ga_1\,\circled{\rm{{\tiny F}}}\,\ga_2 \}\\
&=\|x\|_{\rm max}\,,
\end{align*}
since any representation $\pi'$ of $\ga_1\,\circled{\rm{{\tiny F}}}\,\ga_2$ can be written as  $\pi\circ\theta$.

As for the spatial norm,  we start by recalling that for any pair of even states
$\om_i\in\cs_+(\ga_i)$, $i=1, 2$, the $*$-automorphism $\theta=\theta_1\,\circled{\rm{{\tiny F}}}\,\theta_2$
is unitarily implemented in the representation $\pi_{\om_1\times\om_2}$. More precisely, in light of
Proposition \ref{GNSprod} one has
$$\pi_{\om_1\times\om_2}(\theta(x))= V_{\theta_1, \om_1}\,\circled{\rm{{\tiny F}}}\, V_{\theta_2, \om_2} \left(\pi_{\om_1\times\om_2} (x)\right)   V_{\theta_1, \om_1}\,\circled{\rm{{\tiny F}}}\, V_{\theta_2, \om_2}, \, \, x\in\ga_1\,\circled{\rm{{\tiny F}}}\,\ga_2 $$
where $V_{\theta_i, \om_i}\in \cu(\ch_{\om_i})$ is the self-adjoint unitary that implements $\theta_i$, $i=1, 2$.
But then for $x\in\ga_1\,\circled{\rm{{\tiny F}}}\,\ga_2$ we have
\begin{align*}
&\,\,\,\,\,\,\,\|\theta(x)\|_{\min}\\
&=\sup\{ \|\pi_{\om_1\times\om_2}(\theta(x))\|: \om_i\in\cs_+(\ga_i),\, i=1, 2 \}\\
&=\sup \{\| V_{\theta_1, \om_1}\,\circled{\rm{{\tiny F}}}\, V_{\theta_2, \om_2} \left(\pi_{\om_1\times\om_2} (x)\right)   V_{\theta_1, \om_1}\,\circled{\rm{{\tiny F}}}\, V_{\theta_2, \om_2}\|: \om_i\in\cs_+(\ga_i),\, i=1, 2 \}\\
&=\sup\{  \| \pi_{\om_1\times\om_2} (x) \|:  \om_i\in\cs_+(\ga_i),\, i=1, 2  \}\\
&= \|x\|_{\min}
\end{align*}
and the proof is complete.
\end{proof}

Recall that, for a trivial grading, any abelian $C^*$-algebra is nuclear, namely there is just one way to complete its algebraic tensor product with another $C^*$-algebra. This still holds for nontrivial $\bz_2$-gradings. For the proof, we first consider a couple of technical results.
\begin{lem}
\label{noseparation}
For a (locally) compact Hausdorff space $X$, let $(C(X),\th)$ be a $\bz_2$-graded $C^*$-algebra with non-trivial grading. If $K$ is a proper closed subset of
$\mathcal{E}(\mathcal{S}_+(C(X)))$, then there exists a non-null positive
$f\in C(X)$ such that $\om(f)=0$ for every $\om\in K$.
\end{lem}
\begin{proof}
By Proposition \ref{exevenab}, there exists a proper closed subset $F$  of $X$ such that
$K=\bigg\{\frac{1}{2}(\delta_x+\delta_{\th(x)})\mid x\in F\bigg\}$.
Since $\th^2={\rm id}_X$, in order for $K$ to be a  proper subset,
it is necessary that $G:=F\cup\th(F)$ is still a proper subset of $X$.
But then it is enough to consider a non-null positive function $f\in C(X)$ that
vanishes on $G$ to have the thesis.
\end{proof}

\begin{lem}
\label{factor}
Let $(\ga_1,\th_1)$, $(\ga_2,\th_2)$ be $\bz_2$-graded
$C^*$-algebras with one of the two being commutative.
Let $\beta$ be any compatible $C^*$-norm on  $\ga_1\,\circled{\rm{{\tiny F}}}\,\ga_2$, and take an extreme even state $\om$ on $\ga_1\,\circled{\rm{{\tiny F}}}_\beta\,\ga_2$.
Then there exist $\om_i$ extreme even states on $\ga_i$, $i=1, 2$, such that
$\om$ is the unique extension of $\om_1\times\om_2$.
\end{lem}
\begin{proof}
Suppose $\ga_1$ is commutative, and denote $\ga_\b:=\ga_1\,\circled{\rm{{\tiny F}}}_\beta\,\ga_2$.
If we set $\widetilde{\ga}_1:=\{a_1\otimes \idd_{\ga_2}: a_1\in\ga_1\}$, we see at once
that $\widetilde{\ga}_1$ lies in the centre $\mathcal{Z}(\ga_\beta)$ of
$\ga_\beta$.
We claim that
$$
\om(yx)=\om(y)\om(x)\,, \quad y\in\mathcal{Z}(\ga_\beta)\,,\,\,x\in\ga_\beta\,.
$$
The thesis follows immediately from the claim if we
define $\om_1(a_1):= \om(a_1\otimes \idd_{\ga_2})$, $a_1\in\ga_1$, and
$\om_2(a_2):=\om(\idd_{\ga_1}\otimes a_2)$,  $a_2\in\ga_2$.\\
Indeed, one  has
that $\om\lceil_{\ga_1\,\circled{\rm{{\tiny F}}}\,\ga_2}=\om_1\times\om_2$, which
means both $\om_1$ and $\om_2$ are even thanks to Proposition \ref{converse}.
Furthermore, they must also be extreme among all even states:
 if one of them, say $\om_1$, fails to be extreme, then from
$\om_1=\g\varphi+(1-\g)\psi$ with $0<\g<1$ and $\varphi\neq\psi$ with $\varphi, \psi$ even states, one finds
$\om=\g\varphi\times\om_2+(1-\g)\psi\times\om_2$, contrary
to the assumption that $\om$ is extreme among even states.\\
All is left to do is prove the claim. This can be done by following the proof of Lemma 4.11 in \cite{T1}
word by word when one considers a positive even $y\in\mathcal{Z}(\ga_\beta)$.
If one takes a positive odd $y$, then $\om(y)=0$ and $\om(yx)$ vanishes as well by a standard
application of the Cauchy-Schwarz inequality.
Finally, the case of a possibly non-homogeneous $y$ follows easily by linearity.
\end{proof}

\begin{prop}
\label{nuc}
If $(\ga_1,\th_1)$, $(\ga_2,\th_2)$ are $\bz_2$-graded
$C^*$-algebras and one of the two is commutative, then
every compatible $C^*$-cross norm $\beta$ on $\ga_1\,\circled{\rm{{\tiny F}}}\,\ga_2$
coincides with $\|\cdot\|_{\min}$.
\end{prop}
\begin{proof}
Suppose that $\ga_1$ is commutative. We start by showing that $\|x\|_\beta\leq \|x\|_{\min}$ for every $x\in\ga_1\,\circled{\rm{{\tiny F}}}\,\ga_2$.
Indeed, if $\ga_\beta:=\ga_1\,\circled{\rm{{\tiny F}}}_\beta\,\ga_2$, for any such $x$ one has
\begin{align*}
\|x\|_\beta&= \sup \{\|\pi_\om(x)\|: \om\in \mathcal{E}(\mathcal{S}_+(\ga_\beta))\}\\
& \leq \sup \{\|\pi_{\om_1\times\om_2}(x)\|: \om_i\in \mathcal{E}(\mathcal{S}_+(\ga_i)),\,\, i=1, 2\}\\
&=\|x\|_{\min}\,,
\end{align*}
where the inequality is a straightforward application of Lemma \ref{factor}.

In order to prove the converse inequality, $\|x\|_{\min}\leq \|x\|_\beta$ for every $x\in\ga_1\,\circled{\rm{{\tiny F}}}\,\ga_2$,
we have to show that every product state $\om_1\times\om_2$, with
$\om_i\in\mathcal{E}(\mathcal{S}_+(\ga_i))$, $i=1, 2$, is continuous w.r.t. $\|\cdot\|_\beta$ and thus can be extended to
$\ga_\beta$.
To this end, define
$$
E_\beta:=\{(\om_1, \om_2)\in \mathcal{E}(\mathcal{S_+}(\ga_1))\times\mathcal{E}(\mathcal{S_+}(\ga_2)): \om_1\times\om_2\, {\rm extends\, to}\,\, \ga_\beta\}.
$$
Note that $E_\beta$ is closed in $\mathcal{E}(\mathcal{S_+}(\ga_1))\times\mathcal{E}(\mathcal{S_+}(\ga_2))$, where the latter set is understood
as being equipped with the product of the relative weak* topologies, {\it cf.}  in \cite[Lemma 4.17]{T1}.
We need to prove that $E_\beta=\mathcal{E}(\mathcal{S_+}(\ga_1))\times\mathcal{E}(\mathcal{S_+}(\ga_2))$.
We shall argue by contradiction. Suppose $E_\beta$ is properly contained in $\mathcal{E}(\mathcal{S_+}(\ga_1))\times\mathcal{E}(\mathcal{S_+}(\ga_2))$.
Then there exist (proper) open subsets $U_i\subset\mathcal{E}(\mathcal{S}_+(\ga_i)$, $i=1, 2$, such that
$(U_1\times U_2) \cap E_\beta=\emptyset$.
Consider $K_i:=\mathcal{E}(\mathcal{S}_+(\ga_i))\setminus U_i$, $i=1, 2$.
By Lemma  \ref{noseparation} $K_1$ cannot separate the positive elements of
$\ga_1$, that is there exists a non-zero $a_1\geq 0$ in $\ga_1$ such that
$\om_1(a_1)=0$ for every $\om_1\in K_1$. In particular, if now $a_2$ is any element in
$\ga_2$, we have $\om_1\times\om_2(a_1\otimes a_2)=0$ for every
$(\om_1, \om_2)\in E_\beta$. Thus, by Lemma \ref{factor}, it follows that $\om(a_1\otimes a_2)=0$ for each $\om\in\ce(\cs_+(\ga_\b))$, which contradicts Proposition \ref{extremeven}.
\end{proof}
We finally prove that $\|\cdot\|_{\min}$ on the $\bz_2$-graded tensor product of graded $C^*$-algebras is minimal.
\begin{lem}
\label{factorbis}
Let $(\ga_1,\th_1)$, $(\ga_2,\th_2)$ be $\bz_2$-graded
$C^*$-algebras, and let $\beta$ be any compatible $C^*$-norm on $\ga_1\,\circled{\rm{{\tiny F}}}\,\ga_2$.
If $\om$ is an extreme even state on $\ga_1\,\circled{\rm{{\tiny F}}}_\beta\,\ga_2$
such that the restriction of $\om$ to $\ga_2$ is an extreme even state,
then there exist $\om_i$ extreme even states on $\ga_i$, $i=1, 2$, such
$\om$ is the unique extension of $\om_1\times\om_2$.
\end{lem}
\begin{proof}
The proof can be done in much the same way as in Lemma
\ref{factor}.
\end{proof}

\begin{thm}\label{minimal}
The spatial norm $\|\cdot\|_{\rm min}$ is minimal among all compatible norms
on $\ga_1\,\circled{\rm{{\tiny F}}}\,\ga_2$.
\end{thm}

\begin{proof}
Given a compatible norm $\|\cdot\|_\beta$ on  $\ga_1\,\circled{\rm{{\tiny F}}}\,\ga_2$,
we need to show that for every $x\in\ga_1\,\circled{\rm{{\tiny F}}}\,\ga_2$
one has $\|x\|_{\min} \leq \|x\|_\beta$. This amounts to proving that
for any $\om_i\in\ce(\cs_+(\ga_i))$, $i=1, 2$, the product state
$\om_1\times\om_2$ is bounded w.r.t. the norm $\|\cdot\|_\beta$.
As usual, denote by $\ga_\beta$ the completion of $\ga_1\,\circled{\rm{{\tiny F}}}\,\ga_2$
under the norm $\|\cdot\|_\beta$, and define
$$
E_\beta:=\{(\om_1, \om_2)\in \mathcal{E}(\mathcal{S_+}(\ga_1))\times\mathcal{E}(\mathcal{S_+}(\ga_2)): \om_1\times\om_2\, {\rm extends\, to}\,\, \ga_\beta\}.
$$
Again, $E_\beta$ is a closed subset of $\mathcal{E}(\mathcal{S_+}(\ga_1))\times\mathcal{E}(\mathcal{S_+}(\ga_2))$ in the product of the relative weak* topologies.
By contradiction, suppose $E_\beta$ is properly contained in  $\mathcal{E}(\mathcal{S_+}(\ga_1))\times\mathcal{E}(\mathcal{S_+}(\ga_2))$.\\
Let us denote by $\cp(\cs(\gc))$ the pure states of a $C^*$-algebra $\gc$. If
$$
T_i: \ce(\cs_+(\ga_i))\rightarrow \cp(\cs(\ga_{i, +}))\,, \quad i=1, 2\,,
$$
are the affine homeomorphisms given by the restriction map (as in Proposition \ref{affine}), then
$$
T_1\times T_2:\ce(\cs_+(\ga_1))\times\ce(\cs_+(\ga_2))\rightarrow \cp(\cs(\ga_{1, +}))\times \cp(\cs(\ga_{2, +}))
$$
is an affine homeomorphism. If we set $S_\beta:= (T_1\times T_2) (E_\beta)$, we clearly have that
$S_\beta$ is a proper closed subset of $\cp(\mathcal{S}(\ga_{1, +}))\times \cp(\mathcal{S}(\ga_{2, +}))$.
We can now proceed as in the proof of
\cite[ Lemma 4.18 ]{T1} to find non-zero positive elements $a_i\in\ga_{i,+}$ such that
$\om_1\times\om_2(a_1\otimes a_2)=0$ for any
$(\om_1, \om_2)\in E_\beta$. A contradiction will be
arrived at if we show that the set of states $\{\om_1\times\om_2: (\om_1, \om_2)\in E_\beta\}$
actually separates elements in $\ga_\beta$ of the form $a_1\otimes a_2$ with
$a_1$ and $a_2$ both even and positive.
To this end, let $\widetilde{\ga}_1\subset\ga_{1, +}$ be the unital $C^*$-subalgebra generated by $a_1$.
Let $\rho\in \cp(\cs(\widetilde{\ga}_1))$ such that $\rho(a_1)\neq 0$ and
let $\varphi\in \ce(\cs_+(\ga_2))$ such that $\varphi(a_2)\neq 0$.
The product state $\rho\times\varphi$ is bounded on
$\widetilde{\ga}_1\,\circled{\rm{{\tiny F}}}\ga_2$ with respect to the norm $\|\cdot\|_\beta$
because its restriction to $\widetilde{\ga}_1\,\circled{\rm{{\tiny F}}}\ga_2$ is just the spatial
norm $\|\cdot\|_{\rm min}$ by virtue of Proposition \ref{nuc}.
Let now $\om$ be any extreme even extension of $\rho\times\varphi$ to $\ga_\beta$.
By Lemma \ref{factorbis} $\om$ must be of the form $\om_1\times\om_2$ with
$\om_i\in\ce({\cs_+(\ga_i)})$, $i=1, 2$, and the proof is thus complete.
\end{proof}

\begin{cor}
Given $\bz_2$-graded $C^*$-algebras
$(\ga_i, \theta_i)$, $i=1, 2$, any compatible norm on $\ga_1\,\circled{\rm{{\tiny F}}}\,\ga_2$
is automatically cross.
\end{cor}
\begin{proof}
Let $\|\cdot \|_\beta$ be any norm on $\ga_1\,\circled{\rm{{\tiny F}}}\,\ga_2$ as in the statement.
For any $a_i\in\ga_i$, $i=1, 2$, one clearly has $\|a_1\otimes a_2\|_\beta\leq \|a_1\|\|a_2\|$. On the other hand, by Theorem
\ref{minimal} for
homogeneous $a_i$'s we also have
$\|a_1\otimes a_2\|_\beta\geq \|a_1\otimes a_2\|_{\rm min}=\|a_1\|\|a_2\|$, where the last equality is due to Proposition
\ref{cross}.
\end{proof}

\section*{Acknowledgments}
The authors acknowledge Italian INDAM-GNAMPA.

\end{document}